\documentclass[11pt]{article}
\usepackage{amssymb,amsmath,accents}
\usepackage{amscd}
\usepackage{amsfonts,amsthm,mathrsfs}
\usepackage{setspace} 
\usepackage[dvipdfmx]{graphicx}
\usepackage{here}
\usepackage{latexsym}
\usepackage{color}
\usepackage{authblk}

\def\coe{c_\Omega^\varepsilon}

\newtheorem{theorem}{Theorem}
\newtheorem{lemma}[theorem]{Lemma}
\newtheorem{proposition}[theorem]{Proposition}

\theoremstyle{definition}

\newtheorem{remark}[theorem]{Remark}

\title{\textbf{Extension theorem for $bmo$ in a domain}}

\author{Zhongyang Gu \thanks{zgu@ms.u-tokyo.ac.jp}}
\affil{Graduate School of Mathematical Sciences, The University of Tokyo, 3-8-1 Komaba, Meguro-ku, 153-8914, Tokyo, Japan}
							
\begin{document}
\date{}
\maketitle
\begin{abstract}
The $bmo$ space, also known as the local $BMO$ space, is the $BMO$ space which is uniformly locally $L^1$ in addition.
In this article, we establish an extension theorem for the $bmo$ space defined in an arbitrary uniformly $C^2$ domain. This extension theorem results in a product estimate for $bmo$ functions defined in an arbitrary uniformly $C^2$ domain.
\end{abstract}

\begin{center}
\end{center}

\section{Introduction}
For a function space defined in an open domain $\Omega \subset \mathbf{R}^n$, it is natural to consider the problem if functions of this space can be continuously extended from $\Omega$ to $\mathbf{R}^n$. For example, if $f$ is in $L^p(\Omega)$ with $1 \leq p \leq \infty$, its zero extension $f^{ze} = f \cdot 1_{\Omega}$ naturally belongs to $L^p(\mathbf{R}^n)$ where $1_\Omega$ denotes the characteristic function for domain $\Omega$. Although such extension problem is trivial for $L^p$, the story completely changes when it comes to the space of bounded mean oscillation ($BMO$ for short). In the case for $BMO$, $f \in BMO^\infty(\Omega)$ is not sufficient to have that $f^{ze} \in BMO(\mathbf{R}^n)$. In fact, there exist domains $\Omega$ where bounded linear extension operator from $BMO^\infty(\Omega)$ to $BMO(\mathbf{R}^n)$ does not exist. P.\ W.\ Jones \cite{PJ} gives a necessary and sufficient condition for a domain such that there exists a bounded linear extension operator. 

An open connected subset $D \subset \mathbf{R}^n$ is called a uniform domain if there exists constants $a,b>0$ such that for all $x,y \in D$ there exists a rectifiable curve $\gamma \subset D$ of length $s(\gamma) \leq a|x-y|$ with $\min \left\{s \left(\gamma(x,z)\right), s\left(\gamma(y,z)\right) \right\} \leq b d(z, \partial D)$, where $\gamma(x,z)$ denotes the part of $\gamma$ between $x$ and $z$ on the curve and $d(z, \partial D) = \inf_{w \in \partial D} |z - w|$ denotes the distance from $z$ to the boundary $\partial D$; see e.g. \cite{GO}. 
Let $D \subset \mathbf{R}^n$ be a uniform domain.
Jones' extension theorem guarantees that there is a constant $C_J$ such that for each $f \in BMO^\infty(D)$, there is an extension $\overline{f} \in BMO(\mathbf{R}^n)$ satisfying
\[
[\overline{f}]_{BMO(\mathbf{R}^n)} \leq C_J [f]_{BMO^\infty(D)}
\]
with $C_J$ independent of $f$. The operator $f \mapsto \overline{f}$ is a bounded linear operator. Conversely, if there exists such an extension, then $D$ is a uniform domain.

In \cite{GigaGu2}, a small modification was made to Jones' extension theorem so that we obtained an extension theorem regarding the local BMO space $bmo_\infty^\infty(D) := BMO^\infty(D) \cap L_{\mathrm{ul}}^1(D)$ where
\[
L^1_{\mathrm{ul}}(D) := \left\{ f \in L^1_{\mathrm{loc}} (D) \biggm| \| f \|_{L^1_{\mathrm{ul}}(D)} := \sup_{x \in \mathbf{R}^n} \int_{B_1(x) \cap D} \bigl| f(y) \bigr| \, dy < \infty \right\}.
\]
If $D$ is a uniform domain, the modified Jones' extension theorem says that for $f \in bmo_\infty^\infty(D)$ there exists $\overline{f} \in bmo := BMO \cap L_{\mathrm{ul}}^1(\mathbf{R}^n)$ satisfies
\begin{align} \label{bEE}
\| \overline{f} \|_{bmo(\mathbf{R}^n)} \leq C_J \| f \|_{bmo_\infty^\infty(D)}
\end{align}
with $C_J$ independent of $f$. Moreover, the support of $\overline{f}$ is contained in a small neighborhood of $\overline{D}$. The reason why we are interested in such local $BMO$ spaces ($bmo$) is that multiplication by a H$\ddot{\text{o}}$lder function in such spaces is bounded, i.e., for $\varphi \in C^\gamma(D)$ with $\gamma \in (0,1)$, we have that $\varphi f \in bmo_\infty^\infty(D)$ satisfies the product estimate
\begin{align} \label{bPE}
\| \varphi f \|_{bmo_\infty^\infty(D)} \leq C_J \| \varphi \|_{C^\gamma(D)} \| f \|_{bmo_\infty^\infty(D)}
\end{align}
with $C_J$ independent of $\varphi$ and $f$. 
Because of this multiplication principle, cut-off becomes possible in the space $bmo_\infty^\infty(D)$. The product estimate for $bmo$ follows from the fact that such estimate holds for the local Hardy space $h^1$ and $bmo$ is the dual space of $h^1$, see e.g. \cite[Section 3]{Sa}.

Since the extension theorem and the product estimate for $bmo_\infty^\infty(D)$ relies heavily on the original extension theorem by Jones, we don't know if these results hold or not in the case where $D$ is not a uniform domain. For instance, an aperture domain is an example for a non-uniform domain which is of special interests in fluid mechanics.

Our goal in this article is to establish the extension theorem for $bmo_\infty^\infty(\Omega)$ in the case where $\Omega$ is any arbitrary uniformly $C^2$ domain. We would like to clarify several relevant concepts before we state our main theorem. Let $\Omega \subset \mathbf{R}^n$ be a uniformly $C^2$ domain with $n \geq 2$. Let $\Gamma := \partial \Omega$ denotes the boundary of $\Omega$.
Let $R_0$ be the reach of the boundary $\Gamma = \partial \Omega$. By considering $R_0$ sufficiently small, we may assume that $R_0$ is not only the reach of $\Gamma$ in $\Omega$ but also the reach of $\Gamma$ in $\Omega^c$.
Let $d$ denote the signed distance function from $\Gamma$ which is defined by
\begin{equation*} 
	d(x) = \left \{
\begin{array}{r}
	\inf_{y\in\Gamma}|x-y| \quad\text{for}\quad x\in\Omega, \\
	-\inf_{y\in\Gamma}|x-y| \quad\text{for}\quad x\notin\Omega 
\end{array}
	\right.
\end{equation*}
so that $d(x)=d_\Gamma(x)$ for $x\in\Omega$.
For $0<\rho<R_0$, let $\Gamma_\rho$ be the $\rho$-neighborhood of $\Gamma$ in $\Omega$, i.e.,
\[
\Gamma_\rho = \{ x \in \Omega \mid d_\Gamma(x) < \rho \}
\]
and $\Gamma^\rho$ be the $\rho$-neighborhood of $\Gamma$ in $\mathbf{R}^n$, i.e.,
\[
\Gamma^\rho = \{ x \in \mathbf{R}^n \mid |d(x)| < \rho \}.
\]
We recall the $BMO^\mu$-seminorm for $\mu\in(0,\infty]$ which was defined in \cite{BG}, \cite{BGMST}, \cite{BGS}, \cite{BGST}.
For $f \in L^1_\mathrm{loc}(\Omega)$, we define
\[
	[f]_{BMO^\mu(\Omega)} := \sup \left\{ \frac{1}{\left|B_r(x)\right|} \int_{B_r(x)} \left| f(y) - f_{B_r(x)} \right| \, dy \biggm|
	B_r(x) \subset \Omega,\ r < \mu \right\},
\]
where $f_B$ denotes the average over $B$, i.e.,
\[
	f_B := \frac{1}{|B|} \int_B f(y) \, dy
\]
and $B_r(x)$ denotes the closed ball of radius $r$ centered at $x$ and $|B|$ denotes the Lebesgue measure of $B$.
 The space $BMO^\mu(\Omega)$ is defined as
\[
	BMO^\mu(\Omega) := \left\{ f \in L^1_\mathrm{loc}(\Omega) \bigm|
	[f]_{BMO^\mu} < \infty \right\}.
\]
As in \cite{GigaGu2}, for $\delta \in (0,\infty]$ we set
\[
bmo_\delta^\mu(\Omega) := BMO^\mu(\Omega) \cap L_{\mathrm{ul}}^1(\Gamma_\delta)
\]
with the norm
\[
	\| v \|_{bmo^\mu_\delta} := [v]_{BMO^\mu(\Omega)} + [ v ]_{L^1_{\mathrm{ul}}(\Gamma_\delta)}.
\]
We are now in a position to state our main result.

\begin{theorem} \label{MET}
Let $\Omega \subset \mathbf{R}^n$ be a uniformly $C^2$ domain with $n \geq 2$. There exists $c_\Omega^\ast > 0$ such that for any $\rho \in (0, c_\Omega^\ast)$ and $v \in bmo_\infty^\infty(\Omega)$, there is an extension $\widetilde{v} \in bmo(\mathbf{R}^n)$ such that 
\[
\| \widetilde{v} \|_{bmo(\mathbf{R}^n)} \leq \frac{C}{\rho^n} \| v \|_{bmo_\infty^\infty(\Omega)}
\]
with $C$ independent of $v$ and $\rho$. Moreover, $\operatorname{supp} \widetilde{v} \subset \overline{\Omega_{2 \rho}}$ where
\[
\Omega_{2 \rho} := \{ x \in \mathbf{R}^n \mid d(x, \overline{\Omega}) < 2 \rho \}.
\]
The operator $v \mapsto \widetilde{v}$ is a bounded linear operator.
\end{theorem}

Different from the construction by Jones which delicately deals with the Whitney decomposition of both $\Omega$ and $\Omega^{\mathrm{c}}$, our strategy firstly decomposes $v$ into the sum of $v_1$ and $v_2$ such that the support of $v_1$ is close to $\Gamma$ whereas the support of $v_2$ is away from $\Gamma$. Such decomposition of $v$ is achieved by the multiplication of $v$ with a cut-off function $\theta_\rho$ supported in a small neighborhood of $\Gamma$, i.e., $v_1 := \theta_\rho v$. Since $\Omega$ is not necessarily uniform, at this moment we cannot apply the product estimate that was established for the case of uniform domains to $v_1$ directly. Instead, we apply a localization argument so that we can estimate the $BMO^\rho$-seminorm of $v_1$ in $\Omega$. The key idea of the localization argument is as follow. If a ball $B$ of radius $r(B) \leq \rho$ in $\Omega$ is away from the boundary, then $v_1$ vanishes in this ball. If $B$ is close to the boundary, then we can find a bounded Lipschitz domain $W_\rho$ such that the boundary of $W_\rho$ coincides with $\Gamma$ for a small part and $B \subset W_\rho$. Since $\Gamma$ is uniformly $C^2$, by considering the normal coordinate change in $\Gamma^{R_0}$, we are able to show that the Lipschitz regularity of $\partial W_\rho$ can be uniformly controlled. As $r_{W_\rho} v_1 \in bmo_\infty^\infty(W_\rho)$, we can apply the product estimate to $r_{W_\rho} v_1$ in $W_\rho$. Since a bounded Lipschitz domain is a typical example of a uniform domain and the constant $C_J$ in (\ref{bEE}) and (\ref{bPE}) depends only on the Lipschitz regularity of the domain, we obtain a uniform estimate for $[v_1]_{BMO^\rho(\Omega)}$.

Next, we recall the extension introduced in \cite{GigaGu3} for functions supported in a small neighborhood of $\Gamma$. We extend $v_1$ to $v_1^e$ in $\mathbf{R}^n$ so that $v_1^e$ is even in the direction of $\nabla d$ with respect to $\Gamma$. By considering the normal coordinate change, we then reduce the problem to the half space and prove that $v_1^e \in bmo_\infty^\rho(\mathbf{R}^n)$. Since the $BMO^\infty$-seminorm can be estimated by the $bmo_\infty^\rho$-norm, we thus deduce that $v_1^e \in bmo(\mathbf{R}^n)$. For $v_2$, we simply zero extend it. By a similar argument, it is not hard to show that its zero extension $v_2^{ze} \in bmo(\mathbf{R}^n)$. Setting $\widetilde{v} = v_1^e + v_2^{ze}$ gives us Theorem \ref{MET}.

Since there exists a bounded linear extension operator from $C^\gamma(\Omega)$ to $C^\gamma(\mathbf{R}^n)$ for arbitrary domain $\Omega$, the product estimate for $bmo_\infty^\infty(\Omega)$ follows naturally from Theorem \ref{MET}.

\begin{theorem} \label{PEU}
Let $\Omega \subset \mathbf{R}^n$ be a uniformly $C^2$ domain with $n \geq 2$. Let $\varphi \in C^\gamma(\Omega)$ with $\gamma \in (0,1)$. For each $v \in bmo_\infty^\infty(\Omega)$, the function $\varphi v \in bmo_\infty^\infty(\Omega)$ satisfies
\[
\| \varphi v \|_{bmo_\infty^\infty(\Omega)} \leq C \| \varphi \|_{C^\gamma(\Omega)} \| v \|_{bmo_\infty^\infty(\Omega)}
\]
with $C$ independent of $\varphi$ and $v$.
\end{theorem}

This article is organized as follow. In Section 2, we establish several uniform estimates which are essential for our localization argument. In Section 3, we perform the localization argument to do the cut-off to $v$ and get $v_1$. In Section 4, we extend $v_1$ from $\Omega$ to $\mathbf{R}^n$ and prove Theorem \ref{MET} and Theorem \ref{PEU}. Besides, we apply a similar argument to further obtain an extension theorem for $bmo_\delta^\mu(\Omega)$ in the case where $\delta,\mu<\infty$. In Section 5, we give a simple application of our main extension theorem to construct an example regarding the space $BMO_b^{\infty,\infty}(\Omega)$. In Section 6, we update an extension result that is essential in establishing the Helmholtz decomposition of vector fields of $BMO$ in a domain.

\section{Uniform estimates}
We denote $x' := (x_1, x_2, ... , x_{n-1})$ for $x \in \mathbf{R}^n$ and $\nabla' := (\partial_1, \partial_2, ... , \partial_{n-1})$. Since $\Omega$ is a uniformly $C^2$ domain, there exists $r_\ast, \delta_\ast, L_\Gamma > 0$ such that for each $w_0 \in \Gamma$, up to translation and rotation, there exists a function $\psi_{w_0} \in C^{2}(B_{r_\ast}(0'))$ with
\begin{equation} \label{UBP}
\begin{split}
& | \nabla^k \psi_{w_0} | \leq L_\Gamma \; \text{ in } \, B_{r_\ast}(0') \, \text{ for } \, k=0,1,2, \\
& \nabla' \psi_{w_0}(0') = 0', \, \psi_{w_0}(0') = 0
\end{split}
\end{equation}
such that the neighborhood
\[
U_{r_\ast, \delta_\ast, \psi_{w_0}}(w_0) := \{ (x',x_n) \in \mathbf{R}^{n} \, | \, \psi_{w_0}(x') - \delta_\ast < x_n < \psi_{w_0}(x') + \delta_\ast, \, |x'| < r_\ast \}
\]
satisfies
\[
\Omega \cap U_{r_\ast, \delta_\ast, \psi_{w_0}}(w_0) = \{ (x',x_n) \in \mathbf{R}^{n} \, | \, \psi_{w_0}(x') < x_n < \psi_{w_0}(x') + \delta_\ast, \, |x'| < r_\ast \}
\]
and
\[
\partial \Omega \cap U_{r_\ast, \delta_\ast, \psi_{w_0}}(w_0) = \{ (x',x_n) \in \mathbf{R}^{n} \, | \, x_n = \psi_{w_0}(x'), \, |x'| < r_\ast \}.
\]
For simplicity of explanation, we say that $\Omega$ is of type $(r_\ast, \delta_\ast, L_\Gamma)$.
For $x \in \Omega$, let $\pi x$ be a point on $\Gamma$ such that $|x - \pi x| = d_{\Gamma}(x)$. If $x$ is within the reach of $\Gamma$, then this $\pi x$ is unique. There exists $0< \rho_0 < \mathrm{min} \, \{ r_\ast, \delta_\ast, R_0, 1 \}$ such that for any $w_0 \in \Gamma$,
\begin{align} \label{UW}
U_{\rho_0}(w_0) := \{ x \in U_{r_\ast, \delta_\ast, \psi_{w_0}}(w_0) \, | \, (\pi x)' \in B_{\rho_0}(0'), \, |d(x)| < \rho_0 \}
\end{align}
is contained in $U_{r_\ast, \delta_\ast, \psi_{w_0}}(w_0)$. 

We next consider the normal coordinate in $U_{\rho_0}(w_0)$, i.e.,
\begin{eqnarray} \label{NC0}
x = F(\eta) =
\left\{
\begin{array}{lcl}
\eta' + \eta_n \nabla' d ( \eta', \psi_{w_0}(\eta') ); \\
\psi_{w_0}(\eta') + \eta_n \partial_{x_n} d ( \eta', \psi_{w_0}(\eta') )
\end{array}
\right.
\end{eqnarray}
or shortly
\[
	x = \pi x  - d(x) \mathbf{n} (\pi x).
\]
For each $w_0 \in \Gamma$, $F$ is indeed a local $C^1$-diffeomorphism which maps $V_{\rho_0}$ to $U_{\rho_0}(w_0)$ where $V_{\rho_0} := B_{\rho_0}(0') \times (-\rho_0,\rho_0)$. We indeed have that $F \in C^{1}(V_{\rho_0})$ and $(\nabla_\eta F) (0) = I$. Our first uniform control is for the gradient of $F$ with respect to different $w_0 \in \Gamma$. 

\begin{proposition} \label{UCN}
Let $\Omega \subset \mathbf{R}^n$ be a uniformly $C^2$ domain with $n \geq 2$, $\varepsilon \in (0,1)$. Then there exists a constant $c_\Omega^\varepsilon > 0$, depending on $\Omega$, $n$ and $\varepsilon$ only, such that for any $\rho \in (0,c_\Omega^\varepsilon]$ and $w_0 \in \Gamma$,
\begin{align*}
\| \nabla F - I \|_{L^\infty(V_\rho)} &< \varepsilon, \\
\| \nabla F^{-1} - I \|_{L^\infty(U_\rho(w_0))} &< \varepsilon
\end{align*}
hold simultaneously.
\end{proposition}

\begin{proof}
Let $0 < \varepsilon < 1$ and fix $w_0 \in \Gamma$, $\rho < \rho_0$. By the mean value theorem together with the upper bound of second order derivatives of $\psi_{w_0}$ in (\ref{UBP}), we deduce that 
\begin{align} \label{FDE}
| \nabla' \psi_{w_0} (\eta') | = | \nabla' \psi_{w_0} (\eta') - \nabla' \psi_{w_0} (0') | \leq \| \nabla^2 \psi_{w_0} \|_{L^\infty(B_\rho(0'))} \cdot | \eta' | \leq L_\Gamma \cdot \rho
\end{align}
for any $|\eta'| < \rho$. Since
\begin{align*}
\partial_{\eta_j} x_i &= \delta_{i,j} + \eta_n \cdot \partial_{\eta_j} (\partial_{x_i} d) \\
&= \delta_{i,j} - \eta_n \cdot \frac{\partial_{\eta_j} \partial_{\eta_i} \psi_{w_0}}{(1+ | \nabla' \psi_{w_0} |^2)^{\frac{1}{2}}} + \eta_n \cdot \frac{\sum_{k=1}^{n-1} \partial_{\eta_i} \psi_{w_0} \cdot \partial_{\eta_k} \psi_{w_0} \cdot \partial_{\eta_j} \partial_{\eta_k} \psi_{w_0}}{(1+ | \nabla' \psi_{w_0} |^2)^{\frac{3}{2}}}
\end{align*}
in $V_\rho$ for $1 \leq i,j \leq n-1$, by estimates (\ref{UBP}) and (\ref{FDE}) we have that
\[
| \partial_{\eta_j} x_i (\eta) - \delta_{i,j} | \leq L_\Gamma \rho + (n-1) \cdot (L_\Gamma \rho)^3
\]
for any $\eta \in V_\rho$. By similar calculations, for $\eta \in V_\rho$ we can also deduce that
\[
| \partial_{\eta_n} x_i (\eta) | \leq L_\Gamma \rho
\]
for each $1\leq i \leq n-1$ and
\begin{align*}
| \partial_{\eta_j} x_n (\eta) | &\leq L_\Gamma \rho + (n-1) \cdot (L_\Gamma \rho)^2 \; \; \; \mathrm{for} \; \; \; 1 \leq j \leq n-1, \\
| \partial_{\eta_n} x_n (\eta) - 1 | &\leq (n-1) \cdot (L_\Gamma \rho)^2.
\end{align*} 

Notice that for an invertible matrix $A$, we have that $A^{-1} = \frac{1}{\mathrm{det} (A)} \cdot \mathrm{adj} (A)$ where $\mathrm{adj} (A)$ denotes the adjugate of matrix $A$. Since we have obtained estimates for each entry of $\nabla F$, by considering the inverse of $\nabla F$ we can deduce similar estimates for entries of $\nabla F^{-1}$. Denote $c_{L_\Gamma \rho} := L_\Gamma \rho + (n-1) \cdot (L_\Gamma \rho)^2$. Assume that $L_\Gamma \rho < < 1$, then for any $\eta \in V_\rho$,
\[
(1 - c_{L_\Gamma \rho} )^n - n! \cdot c_{L_\Gamma \rho}^2 \cdot (1 + c_{L \rho})^{n-2} \leq | \mathrm{det} (\nabla F) (\eta) | \leq (1 + c_{L_\Gamma \rho} )^n + n! \cdot c_{L_\Gamma \rho}^2 \cdot (1 + c_{L_\Gamma \rho})^{n-2}.
\]
By considering the adjugate of $\nabla F$, in $V_\rho$ we also have that 
\[
(1 - c_{L_\Gamma \rho})^{n-1} - (n-1)! \cdot c_{L_\Gamma \rho}^2 \cdot (1 + c_{L_\Gamma \rho})^{n-3} \leq | \partial_{x_i} \eta_i (\eta) | \leq (1 + c_{L_\Gamma \rho})^{n-1} + (n-1)! \cdot c_{L_\Gamma \rho}^2 \cdot (1 + c_{L_\Gamma \rho})^{n-3}
\]
for every $1 \leq i \leq n$ and
\[
| \partial_{x_j} \eta_i (\eta) | \leq (n-1)! \cdot c_{L_\Gamma \rho} \cdot (1 + c_{L_\Gamma \rho})^{n-2}
\]
for every $ 1 \leq i,j \leq n$ with $i \neq j$. Therefore, if $\rho$ is chosen to be sufficiently small, then for each $w_0 \in \Gamma$ we can have 
\begin{align*}
\| \nabla F - I \|_{L^\infty(V_\rho)} &< \varepsilon, \\
\| \nabla F^{-1} - I \|_{L^\infty(U_\rho(w_0))} &< \varepsilon
\end{align*}
simultaneously. 

Next we determine how small for $\rho$ is enough. It is easy to see that if $\rho < \mathrm{min} \, 
\{\frac{\varepsilon}{2 L_\Gamma}, \frac{1}{(n-1) L_\Gamma} \}$, we have that $\| \nabla F - I \|_{L^\infty(V_\rho)} < \varepsilon$. Suppose further that $c_{L_\Gamma \rho} < 2 L_\Gamma \rho << 1$, then in $V_\rho$ we have that
\[
1 - c_{L_\Gamma \rho} \cdot (n+1)! \cdot 2^n < | \mathrm{det} (\nabla F) (\eta) | < 1 + c_{L_\Gamma \rho} \cdot (n+1)! \cdot 2^n.
\]
Hence if $2 L_\Gamma \rho < \frac{1}{(n+1)! \cdot 2^{n+1}}$, then
\[
1 - c_{L_\Gamma \rho} \cdot (n+1)! \cdot 2^{n+1} < \frac{1}{| \mathrm{det} (\nabla F) (\eta) |} < 1 + c_{L_\Gamma \rho} \cdot (n+1)! \cdot 2^{n+1}
\]
in $V_\rho$. Since
\[
1 - c_{L_\Gamma \rho} \cdot n! \cdot 2^n < | \partial_{x_i} \eta_i (\eta) | \leq 1 + c_{L_\Gamma \rho} \cdot n! \cdot 2^n,
\]
we deduce that
\[
\left| \frac{1}{| \mathrm{det} (\nabla F) (\eta) |} \cdot \partial_{x_i} \eta_i (\eta) - 1 \right| < c_{L_\Gamma \rho} \cdot ((n+1)!)^2 \cdot 2^{2n+3}
\]
for every $1 \leq i \leq n$ in $V_\rho$. Similar calculations enable us to also obtain that
\[
\left| \frac{1}{| \mathrm{det} (\nabla F) (\eta) |} \cdot \partial_{x_j} \eta_i (\eta) \right| < c_{L_\Gamma \rho} \cdot n! \cdot 2^{n+1}
\]
for every $1 \leq i,j \leq n$ with $i \neq j$ in $V_\rho$. 

Therefore, if $\rho \leq c_\Omega^\varepsilon := \mathrm{min} \, \{ \frac{\varepsilon}{L_\Gamma \cdot ((n+1)!)^2 \cdot 2^{2n+5}}, \frac{\rho_0}{2} \}$, we indeed have 
\begin{align*}
\| \nabla F - I \|_{L^\infty(V_\rho)} &< \varepsilon, \\
\| \nabla F^{-1} - I \|_{L^\infty(U_\rho(w_0))} &< \varepsilon
\end{align*}
simultaneously. 
\end{proof}

We would like to give a uniform estimate, regardless of $w_0 \in \Gamma$,  on the size of the ball centered at $w_0$ that is contained in $U_\rho(w_0)$.

\begin{proposition} \label{BIU}
Let $\varepsilon \in (0,1)$. If $\rho < \mathrm{min} \, \{ \frac{\varepsilon}{8 L_\Gamma}, \rho_0 \}$, then
\[
B_{\rho (1 - \frac{\varepsilon}{2})}(w_0) \subset U_\rho(w_0)
\]
for any $w_0 \in \Gamma$.
\end{proposition}

\begin{proof}
For a ball $B_r(w_0)$ to be contained in $U_\rho(w_0)$, we must have $r \leq \rho$. If $B_r(w_0)$ intersects ${U_\rho(w_0)}^{\mathrm{c}}$ with some $r \leq \rho$, we can find $x \in B_r(w_0)$ of the form $(\eta',h_{w_0}(\eta')) + \tau \nabla d(\eta',h_{w_0}(\eta'))$ with $|\eta'| = \rho$ and $|\tau| \in [0,\rho)$. Notice that
\[
|x - w_0|^2 = |(\eta',h_{w_0}(\eta'))|^2 + \tau^2 + 2 \tau (\eta',h_{w_0}(\eta')) \cdot \nabla d(\eta',h_{w_0}(\eta')).
\]
By the mean value theorem, we can estimate $|\partial_{\eta_i} h_{w_0} (\eta')|$ by $\rho L_\Gamma$ and $|h_{w_0}(\eta')|$ by $\rho^2 L_\Gamma$ for any $|\eta'| \leq \rho$ and $1 \leq i \leq n-1$.
Thus, we deduce that
\[
|(\eta',h_{w_0}(\eta'))|^2 + \tau^2 + 2 \tau (\eta',h_{w_0}(\eta')) \cdot \nabla d(\eta',h_{w_0}(\eta')) \geq \rho^2 + \tau^2 - 4 \tau \rho^2 L_\Gamma
\]
for any $|\eta'| < \rho$. Since $\rho L_\Gamma < \frac{\varepsilon}{8}$, we have that
\[
|x - w_0| \geq \rho \sqrt{1- \frac{\varepsilon}{2}} > \rho(1 - \frac{\varepsilon}{2})
\]
for any $x$ of the form $(\eta',h_{w_0}(\eta')) + \tau \nabla d(\eta',h_{w_0}(\eta'))$ with $|\eta'| = \rho$ and $|\tau| \in [0,\rho)$. Hence for any $w_0 \in \Gamma$, we have that $B_{\rho(1-\frac{\varepsilon}{2})}(w_0) \subset U_\rho(w_0)$.
\end{proof}

Next, we establish a partition of unity for a small neighborhood of the boundary $\Gamma$ in which not only partition functions but also their gradients are uniformly controlled.

\begin{proposition} \label{LFC}
Let $\Omega \subset \mathbf{R}^n$ be a uniformly $C^2$ domain with $n \geq 2$, $\rho \in (0,\frac{\rho_0}{2})$. There exist a countable family of points in $\Gamma$, say $S := \{ x_i \in \Gamma \, \mid \, i \in \mathbf{N} \}$, and a natural number $N_\ast \in \mathbf{N}$ such that
\[
\Gamma^{2 \rho} = \underset{x_i \in S}{\bigcup} \, U_{2 \rho}(x_i)
\]
and for any $x_j \in S$, there exist at most $N_\ast$ points in $S$, say $\{ x_{j_1}, ... , x_{j_{N_\ast}} \} \subset S$, with
\[
U_{2 \rho}(x_j) \cap U_{2 \rho}(x_{j_l}) \neq \emptyset
\]
for each $1 \leq l \leq N_\ast$.
\end{proposition}

\begin{proof}
Let $k_\ast \in \mathbf{N}$ be the smallest integer such that $2^{- k_\ast} \leq \frac{\rho}{\sqrt{n}}$. Let $\mathscr{D}$ be the collection of all dyadic cubes of the form
\[
\{ (y_1, ... , y_n) \in \mathbf{R}^n \, \mid \, m_j 2^{- k_\ast} \leq y_j < (m_j +1) 2^{-k_\ast} \},
\]
where $m_j \in \mathbf{Z}$. Since $\mathscr{D}$ covers the whole space $\mathbf{R}^n$, we can pick out the set of dyadic cubes in $\mathscr{D}$ that intersect the boundary $\Gamma$. Let this subset be denoted by $G = \{ Q_i \in \mathscr{D} \, \mid \, i \in \mathbf{N} \}$ and we have that
\[
\Gamma \subset \underset{i \in \mathbf{N}}{\bigcup} \, Q_i.
\]
We choose $x_i \in Q_i \cap \Gamma$ for each $i \in \mathbf{N}$ and set $S$ to be the set of these points. 

This is indeed the set of points we are seeking. For $y \in \Gamma^{2 \rho}$, there exists $y_0 \in \Gamma$ such that $d(y) = |y-y_0|$. As $G$ covers the boundary $\Gamma$, we have that $y_0 \in Q_j$ for some $j \in \mathbf{N}$. Hence $y \in U_{2 \rho}(x_j)$. We have that
\[
\Gamma^{2 \rho} = \underset{x_i \in S}{\bigcup} \, U_{2 \rho}(x_i).
\]
By the mean value theorem, we can deduce that 
\[
\underset{y \in U_{2 \rho}(x)}{\sup} |y-x| < 5 \rho
\]
for every $x \in \Gamma$. We fix $x_i \in S$. For $Q_j \in G$ with $d(Q_j,Q_i) > 10 \rho$, by the triangle inequality we obviously have that
\[
U_{2 \rho}(x_j) \cap U_{2 \rho}(x_i) = \emptyset.
\]
This means that if $U_{2 \rho}(x_j)$ intersects $U_{2 \rho}(x_i)$, we must have that $d(Q_j,Q_i) \leq 10 \rho$. If $d(Q_j,Q_i) \leq 10 \rho$, then
\[
\underset{y \in Q_j, \, x \in Q_i}{\sup} |y-x| < 12 \rho.
\]
Denote $x_{i_c}$ to be the center of the cube $Q_i$. If $U_{2 \rho}(x_j)$ intersects $U_{2 \rho}(x_i)$, we have that $Q_j \subset Q_i^\ast$ where $Q_i^\ast$ is the cube of side-length $24 \rho$ with center $x_{i_c}$. Since elements of $S$ belong to cubes that do not intersect, we can choose $N_\ast$ to be $24^n \cdot n^{\frac{n}{2}}$.
\end{proof}

Based on $\{ U_{\coe}(x_i) \mid x_i \in S \}$, a locally finite open cover of $\Gamma^{\coe}$, our desired partition of unity for $\Gamma^{\coe}$ can be constructed as follow.

\begin{proposition} \label{PUG}
There exist $\varphi_i \in C^1(\Gamma^{\coe})$ for each $i \in \mathbf{N}$ and a constant $C_U$ such that properties
\begin{equation} 
\begin{split}
& 0 \leq \varphi_i \leq 1 \; \, \text{ for any } \; \, i \in \mathbf{N}, \\
& \operatorname{supp} \varphi_i \subset \overline{U_{\coe}(x_i)} \; \, \text{ for any } \; \, i \in \mathbf{N}, \\
& \displaystyle\sum_{i=1}^\infty \, \varphi_i(x) \equiv 1 \; \, \text{ for any } \; \, x \in \Gamma^{\coe}, \\ 
& \sup_{i \in \mathbf{N}} \, \| \nabla \varphi_i \|_{L^\infty(\Gamma^{\coe})} \leq C_U
\end{split}
\end{equation}
hold.
\end{proposition}

Similar proposition appears in \cite{FKS1}. For the completeness of the theory, we shall provide a proof here.

\begin{proof}[Proof of Proposition \ref{PUG}]
Let us recall an empirical cutoff function that is widely used in various contents, e.g. see \cite[Lemma 2.20 and Lemma 2.21]{JL}.
We consider
\begin{eqnarray*} 
f(t) =
\left\{
\begin{array}{lcl}
\mathrm{exp} ( - \frac{1}{t} ) \quad t > 0, \\
0 \quad \quad \quad \quad \, t \leq 0
\end{array}
\right.
\end{eqnarray*}
and 
\[
\theta(t) := \frac{f(2-t)}{f(t-1) + f(2-t)}
\]
for $t \in \mathbf{R}$.
A simple calculation tells us that $\theta \in C_{\mathrm{c}}^\infty(\mathbf{R})$ with $\theta(t) = 1$ for $|t| \leq 1$ and $\theta(t) = 0$ for $|t| \geq 2$.
For $i \in \mathbf{N}$, we define that
\[
\phi_i(x) := \theta\big( 2 | \big( F^{-1}(x) \big)' | / \coe \big)
\]
for $x \in U_{\coe}(x_i)$ where $F$ in this case is the normal coordinate change between $V_{\coe}$ and $U_{\coe}(x_i)$. By Proposition \ref{LFC}, there exists $S_i := \{ x_{i_1}, x_{i_2}, ... , x_{i_m} \} \subset S$ with $m \leq N_\ast$ and $U_{\coe}(x_{i_l}) \cap U_{\coe}(x_i) \neq \emptyset$ for any $1 \leq l \leq m$. Without loss of generality, we assume that $i_l \neq i$ for each $1 \leq l \leq m$. Then we define $\varphi_i$ in $\Gamma^{\coe}$ by
\begin{eqnarray*} 
\varphi_i(x) :=
\left\{
\begin{array}{lcl}
\frac{\phi_i(x)}{\phi_i(x) + \sum_{l=1}^m \phi_{i_l}(x)} \quad x \in U_{\coe}(x_i), \\
0 \quad \quad \quad \quad \quad x \in \Gamma^{\coe} \setminus U_{\coe}(x_i).
\end{array}
\right.
\end{eqnarray*}

It is trivial to see that $0 \leq \varphi_i \leq 1$ for any $i \in \mathbf{N}$ and
\[
\sum_{i=1}^\infty \varphi_i(x) \equiv 1 \quad \text{in} \quad \Gamma^{\coe}.
\] 
It is sufficient to estimate the gradient of $\varphi_i$. Note that
\[
\partial_j \varphi_i = \frac{\partial_j \phi_i}{\phi_i + \sum_{l=1}^m \phi_{i_l}} - \frac{\phi \cdot ( \partial_j \phi_i + \sum_{l=1}^m \partial_j \phi_{i_l} )}{(\phi + \sum_{l=1}^m \phi_{i_l})^2}.
\]
Let $x \in U_{\coe}(x_i)$ and $\pi x$ be the projection of $x$ in $\Gamma$. By the construction of the set $S$ in the proof of Proposition \ref{LFC}, there exists $x_{i_k} \in S_i$ such that $| \pi x - x_{i_k} | < \frac{\coe}{2}$. This means that $| \big( F^{-1}(x) \big)' | < \frac{\coe}{2}$, i.e., we have that $\phi_{i_k}(x) = 1$. Hence, we deduce that
\[
\phi_i + \sum_{l=1}^m \phi_{i_l} \geq 1 \quad \text{in} \quad U_{\coe}(x_i).
\]
As a result, we have the estimate
\[
| \partial_j \varphi_i | \leq 2 \cdot | \partial_j \phi_i | + \sum_{l=1}^m | \partial_j \phi_{i_l} |.
\]
For any $k \in \mathbf{N}$, we have that
\[
\| \nabla \phi_k \|_{L^\infty(U_{\coe}(x_k))} \leq \frac{C_n}{\rho} \cdot \| \theta' \|_{L^\infty(\mathbf{R})} \cdot \| \nabla F^{-1} \|_{L^\infty(U_{\coe}(x_k))}.
\]
By Proposition \ref{UCN}, we have a uniform estimate for $\| \nabla F^{-1} \|_{L^\infty(U_{\coe}(x_k))}$.
Therefore, combining all estimates together, we finally obtain that
\[
\sup_{i \in \mathbf{N}} \, \| \nabla \varphi_i \|_{L^\infty(\Gamma^{\coe})} \leq \frac{C_{n,N_\ast}}{\rho} \| \theta' \|_{L^\infty(\mathbf{R})}.
\]
\end{proof}

\section{Cut-off}
We consider $v \in bmo_\infty^\infty(\Omega)$. Let $0 < \rho < c_\Omega^\varepsilon/32$ be sufficiently small for which the smallness of $\rho$ will be determined later. For $x \in \Gamma_{\rho_0}^{\mathbf{R}^n}$, we set $\theta_\rho(x) := \theta(d(x)/\rho)$ where $\theta$ is defined in the proof of Proposition \ref{PUG}. Note that $\theta_\rho \in C^2(\mathbf{R}^n)$. We then consider $v_1 := \theta_\rho v$.

\begin{lemma} \label{BNE}
$v_1 \in bmo_\infty^\rho(\Omega)$ satisfies the estimate
\[
\| v_1 \|_{bmo_\infty^\rho(\Omega)} \leq \frac{C}{\rho} \| v \|_{bmo_\infty^\infty(\Omega)}
\]
with $C$ independent of $v$ and $\rho$.
\end{lemma}

Since the domain $\Omega$ is not assumed to be a Jones domain, this lemma cannot be derived by applying the product estimate to $bmo$ functions directly. To establish Lemma \ref{BNE}, we consider a localization argument in which we apply the product estimate to $bmo$ functions locally. For $w_0 \in \Gamma$, we invoke the normal coordinate change $x=F(\eta)$ in $U_{32 \rho}(w_0)$. 
There exists a bounded $C^2$ domain $W$ such that $V_{16} \cap \mathbf{R}_+^n \subset W \subset V_{32} \cap \mathbf{R}_+^n$ and $\partial W \cap \mathbf{R}^{n-1} \times \{0\} = B_{16}(0') \times \{0\}$. Without loss of generality, we assume that $W$ is of type $(\alpha, \beta, L_{\partial W})$ with some constant $L_{\partial W}$.
Let $W_\rho := \{ \rho x \mid x \in W \}$. A simple check tells us that $W_\rho$ is of type $(\alpha \rho, \beta \rho, L_{\partial W}/\rho)$.

\begin{proposition} \label{BRD}
$F(W_\rho)$ is a bounded Lipschitz domain with Lipschitz constant depending on $L_{\partial W}$ only. Moreover, we have that $U_{16 \rho}(w_0) \cap \Omega \subset F(W_\rho) \subset U_{32 \rho}(w_0) \cap \Omega$ and $\partial F(W_\rho) \cap \Gamma = U_{16 \rho}(w_0) \cap \Gamma$.
\end{proposition}

\begin{proof}
Since the normal coordinate change $F$ is a $C^1$-diffeomorphism, we see that $F(W_\rho)$ is a bounded domain which satisfies $F(\partial W_\rho) = \partial F(W_\rho)$. 
Let $\tau_0 \in \partial W_\rho$ and $\delta < \mathrm{min} \, \{ \alpha \rho, \beta \rho, \rho \}$. Without loss of generality we may assume that $\delta = c_0 \rho$ for some sufficiently small universal constant $c_0$. Since $\partial W_\rho$ is uniformly $C^2$, there exist a rotation $R_{\tau_0}$ and $h_{\tau_0} \in C^2(B_\delta(0'))$ such that $\widetilde{\eta_0} := R_{\tau_0} (\eta_0 - \tau_0)$ satisfies
\[
(\widetilde{\eta_0})_n = h_{\tau_0}(\widetilde{\eta_0}')
\]
for any $\eta_0 \in \partial W_\rho$ with $|\widetilde{\eta_0}| < \delta$. Let $y_0 := F(\tau_0)$ and $e_{\tau_0}$ to be the unit normal through $\tau_0$ with respect to boundary $\partial W_\rho$. We set $\tau_n := \tau_0 + \delta e_{\tau_0}$ and $y_n := F(\tau_n)$. There exists another rotation matrix $R_{y_0}$ such that $R_{y_0}(y_n - y_0) = \delta e_n$ where $e_n = (0',1)$.
Let $\zeta_0 \in \partial W_\rho$ such that $|\widetilde{\zeta_0}| < \delta$ where $\widetilde{\zeta_0} := R_{\tau_0} (\zeta_0 - \tau_0)$. We set $x_0 := F(\zeta_0)$ and $z_0 := F(\eta_0)$. In the coordinate system centered at $y_0$ with $y_n$ lying on the $n$-axis in the positive direction, the coordinate of $x_0$ becomes $\widetilde{x_0} := R_{y_0}(x_0 - y_0)$ whereas the coordinate of $z_0$ becomes $\widetilde{z_0} := R_{y_0}(z_0 - y_0)$. By applying the mean value theorem, we have that
\[
(\widetilde{x_0})_n - (\widetilde{z_0})_n = R_{y_0,n} \cdot \int_0^1 (\nabla F) (\eta_0 + t (\zeta_0 - \eta_0)) \, dt \cdot R_{\tau_0}^{-1} \cdot (\widetilde{\zeta_0} - \widetilde{\eta_0})
\]
with $R_{y_0,n}$ denoting the $n$-th row of rotation matrix $R_{y_0}$. Since $(\widetilde{\zeta_0})_n - (\widetilde{\eta_0})_n = h_{\tau_0}(\widetilde{\zeta_0}') - h_{\tau_0}(\widetilde{\eta_0}')$, we deduce that
\begin{align} \label{TXZ}
|(\widetilde{x_0})_n - (\widetilde{z_0})_n| \leq \| \nabla F \|_{L^\infty(V_{16 \rho})} \cdot (1 + \| h_{\tau_0} \|_{L^\infty(B_\delta(0'))}) \cdot |\widetilde{\zeta_0}' - \widetilde{\eta_0}'|.
\end{align}
Applying the mean value theorem again to rewrite $\widetilde{\zeta_0} - \widetilde{\eta_0}$ back to $\widetilde{x_0} - \widetilde{z_0}$, for $1 \leq i \leq n-1$ we have that
\begin{align} \label{EBX}
(\widetilde{\zeta_0})_i - (\widetilde{\eta_0})_i = R_{\tau_0,i} \cdot \int_0^1 (\nabla F^{-1}) (z_0 + t (x_0 - z_0)) \, dt \cdot R_{y_0}^{-1} \cdot (\widetilde{x_0} - \widetilde{z_0})
\end{align}
with $R_{\tau_0,i}$ denoting the $i$-th row of rotation matrix $R_{\tau_0}$.

Fix $1 \leq i \leq n-1$. By deducting the identity matrix $I$ from $\nabla F^{-1}$ in (\ref{EBX}) and then adding $I$ back, we have that
\[
| (\widetilde{\zeta_0})_i - (\widetilde{\eta_0})_i | \leq \| \nabla F^{-1} - I \|_{L^\infty(U_{32 \rho}(w_0))} \cdot | \widetilde{x_0} - \widetilde{z_0} | + | R_{\tau_0,i} \cdot R_{y_0}^{-1} \cdot ( \widetilde{x_0} - \widetilde{z_0} ) |.
\]
In the coordinate system centered at $\tau_0$, there exists $\eta_i \in V_{32 \rho}$ such that $R_{\tau_0} (\eta_i - \tau_0) = \delta e_i$ where $e_i$ denotes the vector whose $j$-th entry equals $\delta_{i,j}$ for each $1 \leq j \leq n$. Hence, $R_{\tau_0,i} = \frac{1}{\delta} (\eta_i - \tau_0)$. Similarly, in the coordinate system centered at $y_0$, we can find $y_i \in U_{32 \rho}(w_0)$ such that $R_{y_0,i} = \frac{1}{\delta} (y_i - y_0)$ where $R_{y_0.j}$ denotes the $j$-th row of $R_{y_0}$ for any $1 \leq j \leq n$. Since $R_{y_0}^{-1} = R_{y_0}^{\mathrm{T}}$, we see that
\[
R_{\tau_0,i} \cdot R_{y_0}^{-1} \cdot (\widetilde{x_0} - \widetilde{z_0}) = (R_{\tau_0,i} - R_{y_0,i}) \cdot R_{y_0}^{\mathrm{T}} \cdot (\widetilde{x_0} - \widetilde{z_0}) + (\widetilde{x_0})_i - (\widetilde{z_0})_i.
\]
Focus on the term that involves $(\widetilde{x_0})_n - (\widetilde{z_0})_n$, characterizations of rows of $R_{\tau_0}$ and $R_{y_0}$ say that
\[
\big( (\widetilde{x_0})_n - (\widetilde{z_0})_n \big) \big( (R_{\tau_0,i} - R_{y_0,i}) \cdot R_{y_0,n} \big) = \frac{\big( (\widetilde{x_0})_n - (\widetilde{z_0})_n \big)}{\delta^2} \big( (\eta_i - y_i) - (\tau_0 - y_0) \big) \cdot (y_n - y_0).
\]

For $\zeta \in V_{32 \rho}$, 
\[
F(\zeta) - \zeta = (0', \psi_{w_0}(\zeta') - \zeta_n) + \zeta_n \cdot (\nabla d)(\zeta', \psi_{w_0}(\zeta')).
\]
An easy check gives that
\[
|\zeta_n \cdot (\partial_{x_j} d)(\zeta', \psi_{w_0}(\zeta'))| \leq | \zeta_n \cdot (\partial_{\zeta_j} \psi_{w_0}) (\zeta') | \leq C_{L_\Gamma} \rho^2.
\]
for $1 \leq j \leq n-1$ and
\[
|\psi_{w_0}(\zeta')| + |\zeta_n| \cdot |((\partial_{x_n} d) (\zeta', \psi_{w_0}(\zeta')) - 1)| \leq C_{L_\Gamma,n} \rho^2.
\]
Hence, for any $\zeta \in V_{32 \rho}$, we have the estimate
\[
|F(\zeta) - \zeta| \leq \frac{C_{L_\Gamma,n}}{c_0^2} \delta^2.
\]
By the mean value theorem, we see that
\begin{align} \label{DPE}
|(y_0 - \tau_0) \cdot (y_n - y_0)| \leq |F(\tau_0) - \tau_0| \cdot |F(\tau_n) - F(\tau_0)| \leq \frac{C_{L_\Gamma,n}}{c_0^2} \cdot \| \nabla F \|_{L^\infty(V_{32 \rho})} \cdot \delta^3.
\end{align}
On the other hand,
\[
|(\eta_i - y_i) \cdot (y_n - y_0)| \leq |(\eta_i - \tau_0) \cdot (y_n - y_0)| + |(\tau_0 - y_0) \cdot (y_n - y_0)| + |(y_0 - y_i) \cdot (y_n - y_0)|.
\]
By decomposing $y_n - y_0$ into $(y_n - \tau_n) + (\tau_n - \tau_0) + (\tau_0 - y_0)$ and applying the estimate (\ref{DPE}), we deduce that
\begin{align*}
|(\eta_i - y_i) \cdot (y_n - y_0)| &\leq |(\eta_i - \tau_0) \cdot (y_n - \tau_n)| + |(\eta_i - \tau_0) \cdot (\tau_0 - y_0)| + |(\tau_0 - y_0) \cdot (y_n - y_0)| \\
&\leq \frac{C_{L_\Gamma,n}}{c_0^2} \cdot (2 + \| \nabla F \|_{L^\infty(V_{32 \rho})}) \cdot \delta^3.
\end{align*}
Therefore,
\[
\left| \big( (\widetilde{x_0})_n - (\widetilde{z_0})_n \big) \big( (R_{\tau_0,i} - R_{y_0,i}) \cdot R_{y_0,n} \big) \right| \leq \frac{C_{L_\Gamma,n}}{c_0^2} \cdot (1 + \| \nabla F \|_{L^\infty(V_{32 \rho})}) \cdot \delta \cdot |(\widetilde{x_0})_n - (\widetilde{z_0})_n|.
\]

If $\rho < c_\Omega^\varepsilon/32$, by Proposition \ref{UCN} we see that
\[
|R_{\tau_0,i} \cdot R_{y_0}^{-1} \cdot (\widetilde{x_0} - \widetilde{z_0})| \leq (n+1) \cdot |(\widetilde{x_0})' - (\widetilde{z_0})'| + \frac{C_{L_\Gamma,n}}{c_0^2} \cdot \delta \cdot |(\widetilde{x_0})_n - (\widetilde{z_0})_n|.
\]
Hence,
\[
| (\widetilde{\zeta_0})_i - (\widetilde{\eta_0})_i | \leq (n+2) \cdot |(\widetilde{x_0})' - (\widetilde{z_0})'| + \big( \frac{C_{L_\Gamma,n}}{c_0^2} \cdot \delta + \varepsilon \big) \cdot |(\widetilde{x_0})_n - (\widetilde{z_0})_n|.
\]
Substitute this estimate back to the inequality (\ref{TXZ}), we obtain that
\[
|(\widetilde{x_0})_n - (\widetilde{z_0})_n| \leq C_{n, L_{\partial W}} |(\widetilde{x_0})' - (\widetilde{z_0})'| + 2n(1 + L_{\partial W}) \cdot \big( \frac{C_{L_\Gamma,n}}{c_0^2} \cdot \delta + \varepsilon \big) \cdot |(\widetilde{x_0})_n - (\widetilde{z_0})_n|. 
\]
Therefore, if we take $\varepsilon < \frac{1}{8n(1+L_{\partial W})}$ and $\rho < \mathrm{min} \, \{ \frac{c_0^2}{8n(1+L_{\partial W}) \cdot C_{L_\Gamma,n}}, \frac{c_\Omega^\varepsilon}{32} \}$, then we have that
\[
|(\widetilde{x_0})_n - (\widetilde{z_0})_n| \leq 2C_{n, L_{\partial W}} |(\widetilde{x_0})' - (\widetilde{z_0})'|.
\]
\end{proof}

Based on this proposition, we have the tool to localize the problem and then to apply the product estimate for $bmo$ functions in a bounded domain.

\begin{proof}[Proof of Lemma \ref{BNE}]
Obviously, the estimate $\| v_1 \|_{L^1(B_1(x) \cap \Omega)} \leq \| v \|_{L^1(B_1(x) \cap \Omega)}$ holds for any $x \in \mathbf{R}^n$. It is sufficient to estimate the $BMO^\rho$-seminorm for $v_1$. Let $r \leq \rho$. For $x \in \Omega$ such that $d(x) \geq 3 \rho$, $v_1 \equiv 0$ in $B_r(x)$ as $B_r(x) \subset \Omega \setminus \overline{\Gamma_{2 \rho}^{\mathbf{R}^n}}$, there is nothing to prove in this case. We then consider $x \in \Omega$ with $d(x) < 3 \rho$ and $B_r(x) \subset \Omega$. Let $\pi x$ be the projection of $x$ on $\Gamma$, i.e., $d(x) = |x-\pi x|$. We have that $B_r(x) \subset U_{8 \rho}(\pi x) \cap \Omega$. By Proposition \ref{BRD}, we see that $B_r(x) \subset F(W_\rho) \subset U_{32 \rho}(\pi x) \cap \Omega$ where $F$ in this case is the normal coordinate change between $U_{32 \rho}(\pi x)$ and $V_{32 \rho}$. Since a bounded Lipschitz domain is a uniform (Jones) domain, we can apply the product estimate for $bmo$ functions \cite[Theorem 13]{GigaGu2} in $F(W_\rho)$, i.e., we have that
\[
\frac{1}{|B_r(x)|} \int_{B_r(x)} | v_1(y) - (v_1)_{B_r(x)} | \, dy \leq \| v_1 \|_{bmo_\infty^\infty(F(W_\rho))} \leq C_0 \| \theta_\rho \|_{C^1(F(W_\rho))} \| v \|_{bmo_\infty^\infty(F(W_\rho))}
\]
where $C_0$ depends only on the Lipschitz constant of $\partial F(W_\rho)$, which is universal by Proposition \ref{BRD}. Therefore, we obtain that
\[
[v_1]_{BMO^\rho(\Omega)} \leq \frac{C_0}{\rho} \| v \|_{bmo_\infty^\infty(\Omega)}.
\]
\end{proof}

Next, let us consider further cut-offs induced by the partition of unity for $\Gamma^{2 \rho}$. For $i \in \mathbf{N}$, we set $v_{1,i} := \varphi_i v_1$ where $\varphi_i$ is the cut-off function defined in Proposition \ref{PUG}.

\begin{lemma} \label{CBE}
$v_{1,i} \in bmo_\infty^\rho(\Omega)$ satisfies the estimate
\[
\| v_{1,i} \|_{bmo_\infty^\rho(\Omega)} \leq \frac{C}{\rho} \| v \|_{bmo_\infty^\infty(\Omega)}
\]
with $C$ independent of $v$ and $\rho$.
\end{lemma}

\begin{proof}
The estimate $\| v_{1,i} \|_{L^1(B_1(x) \cap \Omega)} \leq \| v \|_{L^1(B_1(x) \cap \Omega)}$ is trivial for any $x \in \mathbf{R}^n$. Let $r \leq \rho$. We only need to consider $x \in \Omega$ such that $d(x) < 3 \rho$, $B_r(x) \subset \Omega$ and $B_r(x) \cap U_{2 \rho}(x_i) \neq \emptyset$. 
Proposition \ref{BIU} ensures that if $\varepsilon < \frac{2}{3}$ and $\rho < \frac{1}{4 L_\Gamma}$, then $B_r(x) \subset B_{7 \rho}(x_i) \cap \Omega \subset U_{16 \rho}(x_i) \cap \Omega \subset F(W_\rho)$ where $F$ in this case is the normal coordinate change that maps $V_{32 \rho}$ to $U_{32 \rho}(x_i)$. Again, by applying the product estimate for $bmo$ functions \cite[Theorem 13]{GigaGu2} in $F(W_\rho)$, we have that
\[
\frac{1}{|B_r(x)|} \int_{B_r(x)} | v_{1,i}(y) - (v_{1,i})_{B_r(x)} | \, dy \leq \| v_{1,i} \|_{bmo_\infty^\infty(F(W_\rho))} \leq C_1 \| \varphi_i \|_{C^1(F(W_\rho))} \| v_1 \|_{bmo_\infty^\infty(F(W_\rho))}
\]
with $C_1$ depending only on the Lipschitz constant of $\partial F(W_\rho)$. Note that $bmo_\infty^\infty(F(W_\rho)) = bmo_\infty^\rho(F(W_\rho))$. 
Since $F(W_\rho) \subset U_{32 \rho}(x_i) \cap \Omega \subset \Gamma^{\coe}$, by Proposition \ref{PUG} and Proposition \ref{BRD} we can deduce that
\[
[v_{1,i}]_{BMO^\rho(\Omega)} \leq \frac{C_1 (1+C_U) (1+C_0)}{\rho} \| v \|_{bmo_\infty^\infty(\Omega)}.
\]
\end{proof}

\section{Extension}
\subsection{Extension to a neighborhood of $\Gamma$}
We are now in a position to extend $v_{1,i}$ with respect to the boundary $\Gamma$ for $i \in \mathbf{N}$. 
Let us recall the extension introduced in \cite{GigaGu3}.
For a function $h$ defined in $\Gamma^{\rho_0} \cap \overline{\Omega}$, let $h^e$ denote the even extension of $h$ with respect to $\Gamma$ to $\Gamma^{\rho_0}$ defined by
\begin{align*}
	h^e \left( \pi x + d(x)\mathbf{n}(\pi x) \right)
	= h \left( \pi x - d(x)\mathbf{n}(\pi x) \right)
	\quad \text{for} \quad x \in \Gamma^{\rho_0} \setminus \overline{\Omega}.
\end{align*}
Let $h^o$ denote the odd extension of $h$ with respect to $\Gamma$ to $\Gamma^{\rho_0}$ defined by
\begin{align*}
	h^o \left( \pi x + d(x)\mathbf{n}(\pi x) \right)
	= - h \left( \pi x - d(x)\mathbf{n}(\pi x) \right)
	\quad \text{for} \quad x \in \Gamma^{\rho_0} \setminus \overline{\Omega}.
\end{align*}

\begin{lemma} \label{EE}
Let $\rho < \frac{c_\Omega^\varepsilon}{32}$.
There exists a constant $C$, independent of $v$ and $\rho$, such that the estimate
\[
[v_{1,i}^e]_{bmo(\mathbf{R}^n)} \leq \frac{C}{\rho^n} \| v \|_{bmo_\infty^\infty(\Omega)}
\]
holds for any $i \in \mathbf{N}$.
\end{lemma}

\begin{proof}
It is trivial to see that
\[
\int_{U_{2 \rho}(x_i)} | v_{1,i}^e | \, dy \leq 2 \| \nabla F \|_{L^\infty(V_{2 \rho})} \cdot \| \nabla F^{-1} \|_{L^\infty(U_{2 \rho}(x_i))} \cdot \int_{U_{2 \rho}(x_i) \cap \Omega} | v_{1,i} | \, dy.
\]
Since $\operatorname{supp} v_{1,i} \subset U_{2 \rho}(x_i)$, $\rho < \frac{c_\Omega^\varepsilon}{32}$ implies that 
\[
\| v_{1,i}^e \|_{L^1(\mathbf{R}^n)} \leq 8 \| v_{1,i} \|_{L^1(B_1(x_i) \cap \Omega)} \leq 8 \| v \|_{bmo_\infty^\infty(\Omega)}. 
\]

Since $F(W_\rho)$ is a bounded Lipschitz domain and $v_{1,i} \in bmo_\infty^\infty(F(W_\rho))$, by the extension theorem for $BMO$ functions \cite{PJ}, there exists $v_{1,i}^J \in BMO(\mathbf{R}^n)$ satisfying $r_{F(W_\rho)} v_{1,i}^J = v_{1,i}$ and
\[
[v_{1,i}^J]_{BMO(\mathbf{R}^n)} \leq C [v_{1,i}]_{BMO^\infty(F(W_\rho))}
\]
where by Proposition \ref{BRD} the constant $C$ depends on $L_{\partial W}$ only. Let $c \in \mathbb{R}^n$ be a constant vector. For $B_r(\zeta) \subset V_{16 \rho}^+$, by change of variable $\eta = F^{-1}(y)$ in $V_{16 \rho} = F^{-1}(U_{16 \rho}(x_i))$, we see that
\[
\frac{1}{|B_r(\zeta)|} \int_{B_r(\zeta)} | v_{1,i} \circ F (\eta) - c | \, d \eta \leq \| \nabla F^{-1} \|_{L^\infty(U_{16 \rho}(x_i))} \cdot \frac{1}{|B_r(\zeta)|} \int_{F(B_r(\zeta))} | v_{1,i}(y) - c | \, dy.
\]
Let $x = F(\zeta)$. By Proposition \ref{UCN}, $\rho < \frac{c_\Omega^\varepsilon}{32}$ implies that $\| \nabla F^{-1} \|_{L^\infty(U_{16 \rho}(x_i))} < 2$ and $F(B_r(\zeta)) \subset B_{2r}(x)$. Thus,
\[
\frac{1}{|B_r(\zeta)|} \int_{F(B_r(\zeta))} | v_{1,i}(y) - c | \, dy \leq 2^n \cdot \frac{1}{|B_{2r}(x)|} \int_{B_{2r}(x)} | v_{1,i}^J (y) -c | \, dy.
\]
By considering an equivalent definition of the $BMO$-seminorm, see e.g. \cite[Proposition 3.1.2]{GraM}, we deduce that
\[
[ v_{1,i} \circ F ]_{BMO^\infty(V_{16 \rho}^+)} \leq C_n [v_{1,i}]_{BMO^\infty(F(W_\rho))} \leq \frac{C_n}{\rho} \| v \|_{bmo_\infty^\infty(\Omega)}.
\]
By recalling the results concerning the even extension of $BMO$ functions in the half space, see \cite[Lemma 3.2]{GigaGu} and \cite[Lemma 3.4]{GigaGu}, we can deduce that
\begin{align} \label{EBE}
[ v_{i,n}^e \circ F ]_{BMO^\infty(V_{8 \rho})} \leq \frac{C_n}{\rho} \| v \|_{bmo_\infty^\infty(\Omega)}.
\end{align}

Let $B_r (x)$ be a ball with radius $r \leq \rho$. If $B_r(x) \cap U_{2 \rho}(x_i) = \emptyset$, there is nothing to prove. It is sufficient to consider $B_r(x)$ that intersects $U_{2 \rho}(x_i)$. 
Proposition \ref{BIU} ensures that if $\varepsilon < \frac{1}{4}$, then $B_r(x) \subset B_{7 \rho}(x_i) \subset U_{8 \rho}(x_i)$. By change of variable $y = F(\eta)$ in $U_{16 \rho}(x_i)$, we have that
\[
\frac{1}{|B_r(x)|} \int_{B_r(x)} | v_{1,i}^e (y) - c | \, dy \leq \| \nabla F \|_{L^\infty(V_{16 \rho})} \cdot \frac{1}{|B_r(x)|} \int_{F^{-1}(B_r(x))} | v_{1,i}^e \circ F (\eta) - c | \, d\eta.
\]
Since $F^{-1}(B_r(x)) \subset B_{2r}(\zeta) \subset B_{8 \rho}(0) \subset V_{8 \rho}$, by (\ref{EBE}) we deduce that
\[
\frac{1}{|B_r(x)|} \int_{B_r(x)} | v_{1,i}^e (y) - (v_{1,i}^e)_{B_r(x)} | \, dy \leq \frac{C_n}{\rho} \| v \|_{bmo_\infty^\infty(\Omega)}.
\]
Thus, we obtain that
\[
[v_{1,i}^e]_{BMO^\rho(\mathbf{R}^n)} \leq \frac{C_n}{\rho} \| v \|_{bmo_\infty^\infty(\Omega)}.
\]

For a ball $B$ with radius $r(B) > \rho$, a simple triangle inequality implies that
\[
\frac{1}{|B|} \int_B | v_{1,i}^e (y) - (v_{1,i}^e)_B | \, dy \leq \frac{2}{|B|} \int_B | v_{1,i}^e (y) | \, dy \leq \frac{C_n}{\rho^n} \| v_{1,i}^e \|_{L^1(\mathbf{R}^n)}.
\]
Therefore, we obtain the $BMO$ estimate for $v_{1,i}^e$, i.e.,
\[
[v_{1,i}^e]_{BMO(\mathbf{R}^n)} \leq \frac{C_n}{\rho^n} \| v \|_{bmo_\infty^\infty(\Omega)}.
\]
\end{proof}

Since $\{ U_{2 \rho}(x_i) \mid x_i \in S \}$ is a locally finite open cover of $\Gamma^{2 \rho}$, we are able to estimate the $bmo$ norm for $v_1^e$.

\begin{lemma} \label{BEW}
$v_1^e \in bmo(\mathbf{R}^n)$ satisfies the estimate
\[
\| v_1^e \|_{bmo(\mathbf{R}^n)} \leq \frac{C}{\rho^n} \| v \|_{bmo_\infty^\infty(\Omega)}
\]
with $C$ independent of $v$ and $\rho$.
\end{lemma}

\begin{proof}
Let $r < \rho$ and consider $B_r(x)$ that intersects $\Gamma^{2 \rho}$. By the construction of $S$ in Proposition \ref{LFC}, there exists $x_{i_0} \in S$ such that $|\pi x - x_{i_0}| < \rho$. Thus, by Proposition \ref{BIU} we have that $B_r(x) \subset B_{5 \rho}(x_{i_0}) \subset U_{6 \rho}(x_{i_0})$ as $\varepsilon < \frac{1}{3}$. 
If $x_j \in S$ such that $U_{2 \rho}(x_j) \cap B_r(x) \neq \emptyset$, then $U_{6 \rho}(x_j) \cap U_{6 \rho}(x_{i_0}) \neq \emptyset$. This means that the number of $x_j \in S$ such that $U_{2 \rho}(x_j) \cap B_r(x) \neq \emptyset$ is smaller than the number of $x_j \in S$ such that $U_{6 \rho}(x_j) \cap U_{6 \rho}(x_{i_0}) \neq \emptyset$. Same proof of Proposition \ref{LFC} also shows that for any $x_k \in S$, the number of $x_j \in S$ such that $U_{6 \rho}(x_j) \cap U_{6 \rho}(x_k) \neq \emptyset$ is smaller than some $N_{\ast,0} \in \mathbf{N}$ independent of $x_k$. Hence, we can find at most $N_{\ast,0}$ points in $S$, say $\{ x_{j_1}, ... , x_{j_{N_{\ast,0}}} \} \subset S$, such that $U_{2 \rho}(x_{j_l}) \cap B_r(x) \neq \emptyset$ for each $1 \leq l \leq N_{\ast,0}$.

The $L^1$ norm of $v_1^e$ in $B_r(x)$ is estimated as
\[
\| v_1^e \|_{L^1(B_r(x))} \leq \sum_{l = 1}^{N_{\ast,0}} \| v_{1, j_l}^e \|_{L^1(B_r(x) \cap U_{2 \rho}(x_{j_l}))} \leq 8 N_{\ast,0} \| v \|_{bmo_\infty^\infty(\Omega)}.
\]
Since this estimate holds regardless of $x \in \mathbf{R}^n$, we obtain that
\[
\| v_1^e \|_{L_{\mathrm{ul}}^1(\mathbf{R}^n)} \leq 8 N_{\ast,0} \| v \|_{bmo_\infty^\infty(\Omega)}.
\]
Since 
\[
r_{B_r(x)} v_1^e = \sum_{l=1}^{N_{\ast,0}} r_{B_r(x)} v_{1, j_l}^e,
\]
we have that
\[
\frac{1}{|B_r(x)|} \int_{B_r(x)} | v_1^e (y) - (v_1^e)_{B_r(x)} | \, dy \leq \sum_{l=1}^{N_{\ast,0}} \frac{1}{|B_r(x)|} \int_{B_r(x)} | v_{1, j_l}^e (y) - (v_{1, j_l}^e)_{B_r(x)} | \, dy.
\]
By Lemma \ref{EE}, we deduce that
\[
[v_1^e]_{BMO^\rho(\mathbf{R}^n)} \leq \frac{N_{\ast,0} C_n}{\rho} \| v \|_{bmo_\infty^\infty(\Omega)}.
\]

Let $B$ be a ball in $\mathbf{R}^n$ with radius $r(B) > \rho$. By the triangle inequality,
\[
\frac{1}{|B|} \int_B | v_1^e (y) - (v_1^e)_B | \, dy \leq \frac{2}{|B|} \int_B | v_1^e (y) | \, dy
\]
Let $M \in \mathbb{N}$ be the largest integer such that $M \rho \leq r(B)$. By definition we have that $(M+1) \rho> r(B)$. Note that the ball $B$ is contained in a cube $Q$ of side length $(M+1) \rho$ which shares the same center as $B$. Separating each side of $Q$ equally into $M+1$ parts, we can divide $Q$ equally into $(M+1)^n$ subcubes of side length $\rho$. Hence, we have that
\[
\int_B | v_1^e (y) | \, dy \leq \int_Q | v_1^e (y) | \, dy \leq C_n (M+1)^n \cdot \| v_1^e \|_{L_{\mathrm{ul}}^1(\mathbf{R}^n)}.
\]
Since $r(B) \geq M \rho$, we deduce that
\[
\frac{2}{|B|} \int_B | v_1^e | \, dy \leq \frac{C_n}{\rho^n} \cdot \| v_1^e \|_{L_{\mathrm{ul}}^1(\mathbf{R}^n)}.
\]

Therefore, we finally obtain the estimate
\[
[v_1^e]_{bmo(\mathbf{R}^n)} \leq \frac{N_{\ast,0} C_n}{\rho^n} \| v \|_{bmo_\infty^\infty(\Omega)}.
\]
\end{proof}

\subsection{Extension to $\mathbf{R}^n$}
Let $v_2 := v - v_1$. Note that $\operatorname{supp} v_2 \subset \Omega \setminus \Gamma_\rho$. Let $v_2^{ze}$ denote the zero extension of $v_2$ to $\mathbf{R}^n$, i.e.,
\begin{equation*} 
	v_2^{ze} (x) = \left \{
\begin{array}{r}
	v_2(x) \quad \text{for} \quad x \in \Omega, \\
	0 \quad \text{for} \quad x \notin \Omega. 
\end{array}
	\right.
\end{equation*}
We next estimate the $bmo$ norm of $v_2^{ze}$.

\begin{lemma} \label{IBE}
$v_2^{ze} \in bmo(\mathbf{R}^n)$ satisfies the estimate
\[
\| v_2^{ze} \|_{bmo(\mathbf{R}^n)} \leq \frac{C}{\rho^n} \| v \|_{bmo_\infty^\infty(\Omega)}
\]
with $C$ independent of $v$ and $\rho$.
\end{lemma}

\begin{proof}
Since $r_\Omega v_1^e = v_1$, Lemma \ref{BEW} implies that $v_1 \in bmo_\infty^\infty(\Omega)$ with the estimate
\[
\| v_1 \|_{bmo_\infty^\infty(\Omega)} \leq \frac{C}{\rho^n} \| v \|_{bmo_\infty^\infty(\Omega)}.
\]
Hence, $v_2 = v - v_1 \in bmo_\infty^\infty(\Omega)$ satisfies the estimate
\[
\| v_2 \|_{bmo_\infty^\infty(\Omega)} \leq \frac{C}{\rho^n} \| v \|_{bmo_\infty^\infty(\Omega)}.
\]

Since $v_2^{ze}$ is the zero extension of $v_2$, the estimate $\| v_2^{ze} \|_{L_{\mathrm{ul}}^1(\mathbf{R}^n)} \leq \| v_2 \|_{L_{\mathrm{ul}}^1(\Omega)}$ is trivial. Let $B \subset \mathbf{R}^n$ be a ball with radius $r(B) \leq \rho/2$. If $B$ intersects $\overline{\Omega \setminus \Gamma_\rho}$, then $B \subset \Omega$. In this case, we naturally have that
\[
\frac{1}{|B|} \int_B | v_2^{ze} (y) - (v_2^{ze})_B | \, dy \leq [v_2]_{BMO^\infty(\Omega)}.
\]
If $B \cap \overline{\Omega \setminus \Gamma_\rho} = \emptyset$, then $v_2^{ze} = 0$ in $B$, there is nothing to prove in this case. Hence, we have the estimate
\[
[v_2^{ze}]_{BMO^{\rho/2}(\mathbf{R}^n)} \leq [v_2]_{BMO^\infty(\Omega)} \leq \frac{C}{\rho^n} \| v \|_{bmo_\infty^\infty(\Omega)}.
\]
Let $B \subset \mathbf{R}^n$ be a ball with radius $r(B) > \rho/2$. By same argument in the proof of Lemma \ref{BEW} that decomposes the smallest cube $Q$ containing $B$ into small subcubes of side-length $\rho/2$, we deduce that
\[
\frac{1}{|B|} \int_B | v_2^{ze} (y) - (v_2^{ze})_B | \, dy \leq \frac{2}{|B|} \int_B | v_2^{ze} (y) | \, dy \leq \frac{C}{\rho^n} \| v_2^{ze} \|_{L_{\mathrm{ul}}^1(\mathbf{R}^n)}.
\]

Therefore, we finally obtain that
\[
\| v_2^{ze} \|_{bmo(\mathbf{R}^n)} \leq \frac{C}{\rho^n} \| v \|_{bmo_\infty^\infty(\Omega)}.
\]
\end{proof}

Up till here, we have gathered enough results to prove our main theorem.

\begin{proof}[Proof of Theorem \ref{MET}]
Let 
\begin{align*}
\varepsilon &< \frac{1}{8n(1 + L_{\partial W})}, \\
c_\Omega^\varepsilon &= \mathrm{min} \, \{ \frac{\varepsilon}{L_\Gamma \cdot ((n+1)!)^2 \cdot 2^{2n+4}}, \; \rho_0 \}, \\ 
c_\Omega^\ast &:= \mathrm{min} \, \{ \frac{c_0^2}{16n(1+L_{\partial W}) \cdot C_{L_\Gamma,n}}, \;\frac{c_\Omega^\varepsilon}{64} \}.
\end{align*}
We set $\widetilde{v} := v_1^e + v_2^{ze}$ and let $\rho < c_\Omega^\ast$. An easy check ensures that $\operatorname{supp} \widetilde{v} \subset \overline{\Omega_{2 \rho}}$ and $r_\Omega \widetilde{v} = v$. By Lemma \ref{BEW} and Lemma \ref{IBE}, we see that $\widetilde{v} = v_1^e + v_2^{ze} \in bmo(\mathbf{R}^n)$ satisfies the estimate
\[
\| \widetilde{v} \|_{bmo(\mathbf{R}^n)} \leq \frac{C}{\rho^n} \| v \|_{bmo_\infty^\infty(\Omega)}.
\]
\end{proof}

The product estimate for $v \in bmo_\infty^\infty(\Omega)$ follows directly from the extension theorem.

\begin{proof}[Proof of Theorem \ref{PEU}]
Let $\gamma \in (0,1)$. By \cite[Theorem 13]{GigaGu2}, we see that for $\varphi \in C^\gamma(\Omega)$, there exists $\widetilde{\varphi} \in C^\gamma(\mathbf{R}^n)$ such that $r_\Omega \widetilde{\varphi} = \varphi$ and 
\[
\| \widetilde{\varphi} \|_{C^\gamma(\mathbf{R}^n)} \leq \| \varphi \|_{C^\gamma(\Omega)}.
\]
Extending $v \in bmo_\infty^\infty(\Omega)$ to $\widetilde{v} \in bmo(\mathbf{R}^n)$ by Theorem \ref{MET}, we naturally have that
\[
\| \varphi v \|_{bmo_\infty^\infty(\Omega)} \leq \| \widetilde{\varphi} \widetilde{v} \|_{bmo(\mathbf{R}^n)} \leq C \| \widetilde{\varphi} \|_{C^\gamma(\mathbf{R}^n)} \| \widetilde{v} \|_{bmo(\mathbf{R}^n)} \leq C \| \varphi \|_{C^\gamma(\Omega)} \| v \|_{bmo_\infty^\infty(\Omega)}.
\]
\end{proof}

By almost the same proof of Theorem \ref{MET}, we are able to further establish an extension theorem for $bmo_\delta^\mu(\Omega)$ with $\delta, \mu < \infty$. We recall that $bmo_\infty^\infty(\Omega) \subset bmo_\delta^\mu(\Omega)$ for arbitrary domain $\Omega$ and $\delta, \mu < \infty$ \cite[Theorem 2]{GigaGu2}.

\begin{theorem} \label{EDM}
Let $\Omega \subset \mathbf{R}^n$ be a uniformly $C^2$ domain with $n \geq 2$ and $\mu, \delta \in (0,\infty)$. There exists $c_\Omega^\ast > 0$ such that for any $\rho \in (0, c_\Omega^\ast)$ and $v \in bmo_\delta^\mu(\Omega)$,  there is an extension $\widetilde{v} \in BMO^\mu(\mathbf{R}^n) \cap L_{\mathrm{ul}}^1(\Gamma^\delta)$ such that
\[
[\widetilde{v}]_{BMO^\mu(\mathbf{R}^n)} + [ \widetilde{v} ]_{L_{\mathrm{ul}}^1(\Gamma^\delta)} \leq \frac{C}{\rho} \| v \|_{bmo_\delta^\mu(\Omega)}
\]
with $C$ independent of $v$ and $\rho$. Moreover, $\operatorname{supp} \widetilde{v} \subset \overline{\Omega_{2 \rho}}$ where
\[
\Omega_{2 \rho} := \{ x \in \mathbf{R}^n \mid d(x, \overline{\Omega}) < 2 \rho \}.
\]
The operator $v \mapsto \widetilde{v}$ is a bounded linear operator.
\end{theorem}

\begin{proof}
By \cite[Proposition 1]{GigaGu2}, we see that the space $bmo_{\delta_1}^{\mu_1}(\Omega)$ and the space $bmo_{\delta_2}^{\mu_2}(\Omega)$ are equivalent for any $0 < \delta_1, \delta_2, \mu_1, \mu_2 < \infty$. Without loss of generality, we may assume that $\mu, \delta > c_\Omega^\ast$ where $c_\Omega^\ast$ is defined in the proof of Theorem \ref{MET}. Let $\rho \in (0, c_\Omega^\ast)$. Follow the proofs of Lemma \ref{BNE}, Lemma \ref{CBE}, Lemma \ref{EE} and Lemma \ref{BEW},  we can deduce that $v_1^e \in BMO^\rho(\mathbf{R}^n) \cap L_{\mathrm{ul}}^1(\Gamma^\delta)$ satisfies the estimate
\[
[v_1^e]_{BMO^\rho(\mathbf{R}^n)} + [ v_1^e ]_{L_{\mathrm{ul}}^1(\Gamma^\delta)} \leq \frac{C}{\rho} \| v \|_{bmo_\delta^\mu(\Omega)}.
\]
Moreover, in this case it is trivial that $v_2^{ze} \in BMO^\rho(\mathbf{R}^n) \cap L_{\mathrm{ul}}^1(\Gamma^\delta)$. Still by setting $\widetilde{v} = v_1^e + v_2^{ze}$, we finally obtain that $\widetilde{v} \in BMO^\rho(\mathbf{R}^n) \cap L_{\mathrm{ul}}^1(\Gamma^\delta)$ satisfies the estimate
\[
[\widetilde{v}]_{BMO^\rho(\mathbf{R}^n)} + [ \widetilde{v} ]_{L_{\mathrm{ul}}^1(\Gamma^\delta)} \leq \frac{C}{\rho} \| v \|_{bmo_\delta^\mu(\Omega)} 
\]
with $C$ independent of $v$ and $\rho$.
\end{proof}

\section{Application of the extension theorem}
As defined in \cite{BG}, \cite{BGMST}, \cite{BGS}, \cite{BGST}, we recall a seminorm that controls the boundary behavior.
For $\nu \in (0,\infty]$, we set
\[
	[f]_{b^\nu} := \sup \left\{ r^{-n} \int_{\Omega\cap B_r(x)} \left| f(y) \right| \, dy \biggm|
	x \in \Gamma,\ 0<r<\nu \right\}.
\]
We define the space 
\[
	BMO^{\mu,\nu}_b(\Omega) := \left\{ f \in BMO^\mu(\Omega) \bigm|
	[f]_{b^\nu} < \infty \right\}
\]
with
\[
\| f \|_{BMO_b^{\mu,\nu}(\Omega)} := [f]_{BMO^\mu(\Omega)} + [f]_{b^\nu}.
\]

Let $\mu_0, \nu_0 < \infty$. In \cite[Example 1]{BGST}, we see that there exist examples in $BMO_b^{\mu_0,\nu_0}$, $BMO_b^{\mu_0,\infty}$ and $BMO_b^{\infty,\nu_0}$. By making use of the extension theorem and the product estimate established in this article, we shall give an example of a function that belongs to $BMO_b^{\infty,\infty}$ but does not belong to $L^\infty$.

We consider the case where the domain $\Omega$ is the half space $\mathbf{R}_+^2$. Let $f =\mathrm{log} \, x_2$ defined in the layer domain $D_L := \{ 0 < x_2 < 1 \}$. For a cube $Q = [a, a+1] \times [b, b+1]$ that intersects $D_L$, we have that
\[
\int_{Q \cap D_L} | \mathrm{log} \, x_2 | \, dx = - \int_0^{b+1} \mathrm{log} \, x_2 \, dx_2 \leq 1.
\]
Hence, we see that $f \in bmo_\infty^\infty(D_L)$. By Theorem \ref{MET}, we can find $\widetilde{f} \in bmo(\mathbf{R}^2)$ such that $r_{D_L} \widetilde{f} = f$ and $\operatorname{supp} \widetilde{f} \subset \{ -1 < x_2 < 2 \}$. Set $\widetilde{g} (x_1, x_2) := \widetilde{f} (x_1, x_2 - 2)$ for any $x = (x_1, x_2) \in \mathbf{R}^2$ and $g := r_{\mathbf{R}_+^2} \widetilde{g}$. Note that $\operatorname{supp} g \subset \{ 1 < x_2 < 4 \}$.

\begin{proposition} \label{III}
$g \in BMO_b^{\infty,\infty}(\mathbf{R}_+^2)$ but $g \notin L^\infty(\mathbf{R}_+^2)$.
\end{proposition}

\begin{proof}
It is trivial to see that $g \in BMO^\infty(\mathbf{R}_+^2)$ and $g \notin L^\infty(\mathbf{R}_+^2)$. We only need to estimate the $b^\infty$-norm for $g$. Since $\operatorname{supp} g \subset \{ 1 < x_2 < 4 \}$, it is sufficient to estimate
\[
\frac{2}{|Q_r(x)|} \int_{Q_r(x) \cap \mathbf{R}_+^2} |g| \, dy
\]
for $r \geq 1$ and $x = (x_1,0) \in \partial \mathbf{R}_+^2$ where $Q_r(x)$ denotes the square with center $x$ of side-length $2r$. Without loss of generality, we may assume that $g$ is only a function of $x_2$. Hence, a direct calculation shows that
\[
\frac{2}{|Q_r(x)|} \int_{Q_r(x) \cap \mathbf{R}_+^2} |g| \, dy = \frac{1}{2r^2} \int_{x_1 - r}^{x_1 + r} \int_1^r |g| \, dy_2 \, dy_1 \leq 2 \int_0^1 | \mathrm{log} \, z_2 | \, dz_2 \leq 2.
\]
\end{proof}

\begin{remark} \label{EB2}
Let $\phi \in C_{\mathrm{c}}^\infty(B_8(0))$ with $\phi \equiv 1$ in $B_6(0)$, by Proposition \ref{III} we see that $\phi g \in BMO_b^{\infty,\infty}(\mathbf{R}_+^2) \cap L^2(\mathbf{R}_+^2)$ but $\phi g \notin L^\infty(\mathbf{R}_+^2)$.
\end{remark}

\section{Extension of vector fields in $bmo$ in a domain}
Note that Lemma \ref{EE} basically coincide with \cite[Proposition 2]{GigaGu3} in the statement. However, the proof of Lemma \ref{EE} involves the localization argument in this article, which actually improves \cite[Proposition 2]{GigaGu3} in the sense that \cite[Proposition 2]{GigaGu3} holds for any uniformly $C^2$ domain instead of just for bounded domain. Here we provide an update of \cite[Proposition 2]{GigaGu3}. 

We consider the space
\[
vbmo(\Omega) := \{ u \in bmo_\infty^\infty(\Omega) \, | \, [ \nabla d \cdot u ]_{b^\nu} < \infty \}
\] 
equipped with the norm
\[
\| u \|_{vbmo(\Omega)} := \| u \|_{bmo_\infty^\infty(\Omega)} + [ \nabla d \cdot u]_{b^\nu}.
\]
This space is independent of $\nu \in (0,\infty]$.
Let $u \in vbmo(\Omega)$. We set $u_1 = \theta_\rho u$, $u_{1,i} = \varphi_i u_1$.
Let $P u_{1,i}^o := (\nabla d \cdot u_{1,i}^o) \nabla d$ denotes the normal component of $u_{1,i}^o$ whereas $Q u_{1,i}^e := u_{1,i}^e - (\nabla d \cdot u_{1,i}^e) \nabla d$ denotes the tangential component of $u_{1,i}^e$.

\begin{lemma} \label{OEV}
Let $\rho < \frac{c_\Omega^\varepsilon}{48}$.
There exists a constant $C$, independent of $v$ and $\rho$, such that the estimates
\begin{align*}
[P u_{1,i}^o]_{bmo(\mathbf{R}^n)} &\leq \frac{C}{\rho^n} \| u \|_{vbmo(\Omega)}, \\
[\nabla d \cdot P u_{1,i}^o]_{b^\infty(\Gamma)} &\leq \frac{C}{\rho^n} \| u \|_{vbmo(\Omega)}
\end{align*}
hold for any $i \in \mathbf{N}$ and $\nu \in (0,\infty]$.
\end{lemma}

\begin{proof}
Follow the proofs of \cite[Proposition 2]{GigaGu3} and Lemma \ref{EE}, we are done.
\end{proof}

\begin{lemma} \label{OES}
$P u_1^o \in bmo(\mathbf{R}^n)$ satisfies the estimates
\begin{align*}
\| P u_1^o \|_{bmo(\mathbf{R}^n)} &\leq \frac{C}{\rho^n} \| u \|_{vbmo(\Omega)}, \\
[\nabla d \cdot P u_1^o]_{b^\infty(\Gamma)} &\leq \frac{C}{\rho^n} \| u \|_{vbmo(\Omega)}
\end{align*}
with $C$ independent of $u$ and $\rho$.
\end{lemma}

\begin{proof}
Follow the proof of Lemma \ref{BEW}, we are done.
\end{proof}

Similar as in \cite[Proposition 2]{GigaGu3}, we set 
\[
\overline{u_1} := P u_1^o + Q u_1^e.
\]
By Lemma \ref{BEW}, we have that $\overline{u_1} \in bmo(\mathbf{R}^n)$. 
Let $u_2 := u - u_1$ and $u_2^{ze}$ be the zero extension of $u_2$ to $\mathbf{R}^n$.
Since $\overline{u_1}$ coincide with $u_1$ in $\Omega$, following the proof of Lemma \ref{IBE} we can show that $u_2^{ze} \in bmo(\mathbf{R}^n)$ satisfying 
\[
\| u_2^{ze} \|_{bmo(\mathbf{R}^n)} \leq \frac{C}{\rho^n} \| u \|_{vbmo(\Omega)}
\]
with $C$ independent of $u$ and $\rho$. Therefore, by setting $\overline{u} := \overline{u_1} + u_2^{ze}$, we obtain an extension of $u$ whose normal component in a small neighborhood of $\Gamma$ is odd with respect to $\Gamma$ whereas the tangential component in a small neighborhood of $\Gamma$ is even with respect to $\Gamma$. We summarize the extension theorem for a vector field of $bmo$ in a domain as follow.

\begin{theorem} \label{VET}
Let $\Omega \subset \mathbf{R}^n$ be a uniformly $C^2$ domain with $n \geq 2$. There exists $c_\Omega^{\ast \ast} > 0$ such that for any $\rho \in (0,c_\Omega^{\ast \ast})$ and $u \in vbmo(\Omega)$, there is an extension $\overline{u} \in bmo(\mathbf{R}^n)$ such that
\[
\| \overline{u} \|_{bmo(\mathbf{R}^n)} + [\nabla d \cdot \overline{u}]_{b^\infty(\Gamma)} \leq \frac{C}{\rho^n} \| u \|_{vbmo(\Omega)}
\]
with $C$ independent of $u$ and $\rho$. Moreover, $\operatorname{supp} \overline{u} \subset \overline{\Omega_{2 \rho}}$ where
\[
\Omega_{2 \rho} := \{ x \in \mathbf{R}^n \mid d(x, \overline{\Omega}) < 2 \rho \}.
\]
The operator $u \mapsto \overline{u}$ is a bounded linear operator.
\end{theorem}

The constant $c_\Omega^{\ast \ast}$ can be taken as 
\[
c_\Omega^{\ast \ast}:= \mathrm{min} \, \{ \frac{c_0^2}{16n(1+L_{\partial W}) \cdot C_{L_\Gamma,n}}, \;\frac{c_\Omega^\varepsilon}{96} \}.
\]

\end{document}